\documentclass[12pt,oneside]{amsart} 
\usepackage{geometry} 
\usepackage{color}

\usepackage{amssymb}

\theoremstyle{plain}
\newtheorem{theorem}{Theorem}[section]     
     
\newtheorem{lemma}[theorem]{Lemma}
\newtheorem{cor}[theorem]{Corollary}
\newtheorem{ex}[theorem]{Example}
\newtheorem{proposition}[theorem]{Proposition}
\newtheorem{define}[theorem]{Definition}

\theoremstyle{definition}
\newtheorem{remark}[theorem]{Remark}
\newtheorem{conv}[theorem]{Convention}

\newenvironment{proofof}[1]{\noindent{\it Proof of
#1.}}{\hfill$\square$\\\mbox{}}

\newcommand{\ord}{\mathop{\rm ord}\nolimits}
\newcommand{\rank}{\mathop{\rm rank}\nolimits}
\newcommand{\md}{\mathop{\rm mod}\nolimits}
\newcommand{\supp}{\mathop{\rm supp}\nolimits}

\usepackage{url}

\begin{document}

\title[Separating Noether number]{The separating Noether number of the direct sum of several copies of a cyclic group}
\author[Barna Schefler]{Barna Schefler}
\address{E\"otv\"os Lor\'and University, 
P\'azm\'any P\'eter s\'et\'any 1/C, 1117 Budapest, Hungary} 
\email{scheflerbarna@yahoo.com}
\thanks{Partially supported by the Hungarian National Research, Development and Innovation Office,  NKFIH K 138828.}

\maketitle

\begin{abstract} 
The exact value of the separating Noether number of some finite abelian groups is determined, including the direct sums of cyclic groups of the same order. 
\end{abstract}

\section{Introduction} \label{sec:intro} 

A long-standing open question in arithmetic combinatorics is whether the \emph{Davenport constant} (i.e. the maximal length of an irreducible zero-sum sequence) over the direct sum $C_n^r$ of $r$ copies of the cyclic group of order $n$ equals $1+(n-1)r$ (see \cite{gaogeroldinger1}, \cite{glp}, \cite{girard}, \cite{mr}). Equivalently, the question is whether  
the \emph{Noether number} (i.e. the general upper bound for the degrees of the generators of the algebras of polynomial invariants) of $C_n^r$ equals $1+(n-1)r$. 
   
In recent years the study of separating systems of polynomial invariants 
(as a relaxation of the concept of generating systems) received much attention. In particular, in analogy with the Noether number, the notion of the \emph{separating Noether number} of a finite group was introduced in \cite{kohls-kraft}. In the present paper we shall determine the exact value of the separating Noether number of $C_n^r$, by 
proving certain results on zero-sum sequences over finite 
abelian groups. 

Given a finite dimensional $\mathbb{C}$-vector space $V$ endowed with an action of a finite group $G$ via linear transformations, we shall write 
$\mathbb{C}[V]^G$ for the algebra of polynomial functions $V\to \mathbb{C}$ that are constant on $G$-orbits. This is a graded subalgebra of the 
algebra $\mathbb{C}[V]$ of polynomial functions on $V$ 
(which can be identified with the $n$ variable polynomial algebra over $\mathbb{C}$ where $n=\dim_{\mathbb{C}}(V)$, equipped with the standard grading). 
It is known that if $v,w\in V$ have different $G$-orbits, then there exists a homogeneous polynomial invariant $f\in \mathbb{C}[V]^G$ of degree at most $|G|$ with $f(v)\neq f(w)$, see \cite[Theorem 3.12.1]{derksen-kemper}. 

\begin{define}\label{sepnoethnumdef}
    Write $\beta_{sep}(G,V)$ for the minimal positive integer $d$ such that for any $v,w\in V$ with different $G$-orbits, 
    there exists a homogeneous $f\in \mathbb{C}[V]^G$ of degree at most $d$ with $f(v)\neq f(w)$. The separating Noether number $\beta_{sep}(G)$ of a finite group $G$ is defined by
\begin{center}
    $\beta_{sep}(G):=\underset{V}{\sup}\{\beta_{sep}(G,V): V$ is a finite dimensional representation of  $G\}$
\end{center}
\end{define}

The above definition (which appeared first in \cite{kohls-kraft}) is motivated by the Noether number $\beta(G)$, introduced in \cite{schmid} 
as the maximal degree in a minimal homogeneous generating system of the algebra $\mathbb{C}[V]^G$, as $V$ ranges over all possible finite dimensional $G$-modules. By the famous theorem of Noether, we have 
$\beta(G)\le |G|$ (cf. \cite{noether}). Note also the obvious inequality 
\[\beta_{sep}(G)\le\beta(G).\]

In this paper the focus will be on determining the separating Noether numbers of some infinite families of abelian groups. 
Our main result is the following: 

\begin{theorem}\label{ptlns}
 For positive integers $n\ge 2$ and $r$ denote by $C_n^r$ the direct sum $C_n\oplus\cdots\oplus C_n$ of $r$ copies of the cyclic group $C_n$ of order $n$, and let $p$ be the minimal prime divisor of $n$. 
 Then we have \textbf{}
 \[\beta_{sep}(C_n^r)=\begin{cases}
        ns,&\mbox{ if }r=2s-1 \mbox{ is odd} \\
        ns+\frac{n}{p},&\mbox{ if }r=2s \mbox{ is even}.\\
    \end{cases}\]
 \end{theorem}

In fact Theorem~\ref{ptlns} is a special case of our 
Theorem~\ref{ptlci} and Theorem~\ref{psci}. Note that in sharp contrast with Theorem~\ref{ptlns}, the exact values of the Noether numbers (equivalently Davenport constants) of the groups $C_n^r$ are not known, despite much interest in determining them (\cite{gaogeroldinger1}, \cite{glp}, \cite{girard}, \cite{mr}).\par
Recall that the \emph{Davenport constant} 
$\mathsf{D}(G)$ of a finite abelian group $G$ is the maximal length of an atom (i.e. irreducible element) in the monoid of zero-sum sequences over $G$.
The most common examples of infinite families of abelian groups for which the exact values of the Davenport constants are known are the abelian groups of rank two and abelian $p$-groups (for such groups $G$ we have $\mathsf{D}(G)=\mathsf{d}^*(G)+1$) 
(see \eqref{eq:d*G} for the definition of $\mathsf{d}^*(G)$). The determination of the separating Noether number of 
these groups seems to be more laborious or difficult than the proof of Theorem~\ref{ptlns}. 
In a parallel paper \cite{schefler} we are discussing the exact value of the separating Noether number of abelian groups of rank two and three.\par

 It has long been known that for an abelian group, the values of the Noether number and the Davenport constant coincide:
\begin{equation}\label{b=d}
    \beta(G)=\mathsf{D}(G)
\end{equation}
(see \cite{schmid}; see also  \cite{cziszter-domokos-geroldinger} for further discussion of this connection). The separating Noether number of a finite abelian group $G$ also has an interpretation in terms of zero-sum sequences over $G$ that we shall explain in Section~\ref{sec:separating-Davenport}.  In particular, in this part we define the concept of \emph{group atom} (see  Definition \ref{groupatomdef}), and recall its role in the study of the separating Noether number. In Section \ref{sec:upperbound} we give a general upper bound for the separating Noether number of an abelian group. In Section \ref{sec:groupatoms} we prove some lemmas that will help us in detecting group atoms. By using these lemmas we construct certain group atoms in Section \ref{sec:lowerbound}. 
In Section \ref{sec:results} we 
deduce Theorem~\ref{ptlci} and 
Theorem~\ref{psci}, which give the separating 
Noether number for a class of finite abelian groups, including the groups $C_n^r$.  Theorem~\ref{ptlns} follows therefore as a special case. \par

\section{The separating Noether number of an abelian group}\label{sec:separating-Davenport}
Let $g_1,\dots,g_k$ be distinct elements of the (additively written) finite abelian group $G$. Then 
\begin{center}
    $\mathcal{G}(g_1,\dots,g_k):=\{[m_1,\dots,m_k]\in\mathbb{Z}^k:\sum m_ig_i=0\in G\}$\\
\end{center}
is a subgroup of the additive group of $\mathbb{Z}^k$. Its submonoid 
\begin{center}
    $\mathcal{B}(g_1,\dots,g_k):=\mathbb{N}^k\cap \mathcal{G}(g_1,\dots,g_k)$
\end{center}   
is a finitely generated Krull monoid, playing an essential role in the 
multiplicative ideal theory of Krull domains, see \cite[Chapter 2]{ghk}). These monoids are also called \emph{block monoids}, see for example \cite[Definition 2.5.5]{ghk}.  
It is not hard to see that  $\mathcal{B}(g_1,\dots,g_k)$ generates $\mathcal{G}(g_1,\dots,g_k)$ as a group. \par 
An element of the monoid $\mathcal{B}(g_1,\dots,g_k)$ is an \emph{atom}, if it can not be written as the sum of two non-zero elements of the monoid. The \emph{length} of an element $\mathsf{m}=[m_1,\dots,m_k]\in\mathcal{B}(g_1,\dots,g_k)$ is  $|\mathsf{m}|=\sum_{i=1}^k m_i$. 
Denote by $\mathsf{D}(\mathcal{B}(g_1,\dots,g_k))$ the maximal length of an atom in the monoid $\mathcal{B}(g_1,\dots,g_k)$. The support of an element is defined as:
\begin{center}
    $\supp(\mathsf{m})=\{i\in\{1,\dots,k\}: m_i\neq 0\}$
\end{center}
The size of the support is $|\supp(\mathsf{m})|$. Let $\mathsf{e_i}$ be the $i^{th}$ standard basis vector of $\mathbb{Z}^k$, and $\ord(g_i)$ the order of $g_i$ in $G$. Of course, $\ord(g_i)\mathsf{e_i}\in \mathcal{B}(g_1,\dots,g_k)$, so $\sum_{i=1}^k \ord(g_i)\mathsf{e_i}\in\mathcal{B}(g_1,\dots,g_k)$. If $0\leq m_i\leq\ord(g_i)$ for any $i$, then we define the \emph{complementer} of $\mathsf{m}$: 
\begin{equation}\label{eq:m*}
\mathsf{m^*}=\sum_{i=1}^k \ord(g_i)\mathsf{e_i}-\mathsf{m}\in\mathcal{B}(g_1,\dots,g_k)\end{equation} 

In the special case when $g_1,\dots,g_k$ is an  enumeration of all the elements of $G$, then by slight abuse of notation we write $\mathcal{B}(G):=\mathcal{B}(g_1,\dots,g_k)$ and 
$\mathsf{D}(G):=\mathsf{D}(\mathcal{B}(G))$; 
clearly, the isomorphism type of the monoid endowed with the length function does not depend on the ordering of the elements of $G$.  \par

\begin{remark}
The monoid $\mathcal{B}(G)$ can be interpreted as the monoid of zero-sum sequences over $G$. 
Namely, associate with $\mathsf{m}\in \mathcal{B}(G)=\mathcal{B}(g_1,\dots,g_k)$ the sequence over $G$ containing $g_i\in G$ with multiplicity $m_i$. 
For a survey on zero-sum sequences over abelian groups see  \cite{d^*}, \cite[Chapter 5]{ghk}, \cite[Chapter 10]{grynkiewicz}. 
\end{remark}

\begin{conv}
Throughout this article the following convention will be used: $G$ will stand for the abelian group $G=C_{n_1}\oplus C_{n_2}\oplus \dots \oplus C_{n_r}$ where $2\le n_r\mid\dots\mid n_2\mid n_1$, and $r=2s$ or $2s-1$, depending on its parity, $g_1,\dots,g_k$ are fixed distinct elements of $G$, and $\mathcal{B}(g_1,\dots,g_k)$ is the corresponding block monoid. Note that here $r$ is the \emph{rank} (i.e. the minimal number of generators), and $n_1=\exp(G)$ is the \emph{exponent} (i.e. the least common multiple of the orders  of the elements) of the group $G$. 
\end{conv}
\noindent
Setting  
\begin{equation}\label{eq:d*G}
    \mathsf{d}^*(G):=\sum_{i=1}^r(n_i-1), 
\end{equation}
we have the inequality 
\begin{equation}\label{lowerbound}
    \mathsf{d}^*(G)+1\leq\mathsf{D}(G)
\end{equation}
(see for example Section 3 in \cite{d^*}).
Results from arithmetic combinatorics show that for p-groups (see \cite{p-group}) or for groups of rank two (see \cite{rank2}) equality holds in \eqref{lowerbound}. There are some infinite families of abelian groups for which the inequality is known to be strict 
(see for example \cite{geroldinger-schneider}). In general it is not known when equality holds.\par
In order to state the analogues of \eqref{b=d} for the separating 
Noether number we need 
the following definition: 
\begin{define}\label{groupatomdef}
An element $\mathsf{m}\in\mathcal{B}(g_1,\dots,g_k)$ is called a \emph{group atom in 
$\mathcal{B}(g_1,\dots,g_k)$}  
if $\mathsf{m}$ can not be written as an \emph{integral} linear combination of elements 
of $\mathcal{B}(g_1,\dots,g_k)$ that have length strictly smaller than $|\mathsf{m}|$. 
\end{define}

\begin{remark}
     If $\mathsf{m}$ is a group atom in the monoid $\mathcal{B}(g_1,\dots,g_k)$, then it is also an atom in $\mathcal{B}(g_1,\dots,g_k)$, hence in particular we have the inequalities $m_i\leq \ord(g_i)$ for all $i=1,\dots,k$. So its complementer,  $\mathsf{m}^*\in \mathcal{B}(g_1,\dots,g_k)$ is defined  as in \eqref{eq:m*}.
\end{remark}
 
The relation between the separating Noether number of an abelian group and the theory of zero-sum sequences  
is established in \cite{domokos_abelian} (building on \cite{domokos_typical} and \cite{domokos-szabo}) by the following statement: 

\begin{lemma}\label{md2}  {\rm{(Corollary 2.6. in \cite{domokos_abelian})}}
The number $\beta_{sep}(G)$ is the maximal length of a group atom in $\mathcal{B}(g_1,\dots,g_k)$, where $\{g_1,\dots,g_k\}$ ranges over all subsets of size $k\leq \rank(G)+1$ of the abelian group $G$.
\end{lemma}
\noindent 
 For related results on not necessarily finite abelian groups see \cite{domokos:BLMS}. 

\section{A general upper bound}\label{sec:upperbound}
\begin{lemma}\label{m*geqm}
Let $\mathsf{m}$ be a group atom in the monoid $\mathcal{B}(g_1,\dots,g_k)$ such that $|\mathsf{m}|> \max\{\ord(g_i)\mbox{ for } i=1,\dots,k\}$. Then we have the following inequalities: 
\begin{itemize} 
\item[(i)]  
$|\mathsf{m}|\le |\mathsf{m}^*|$. 
\item[(ii)] $2|\mathsf{m}|\le \sum_{i=1}^k\ord(g_i)$. 
\end{itemize}
\end{lemma}
\begin{proof} Both (i) and (ii) follow from the equality 
     $\mathsf{m}=\sum_{i=1}^k \ord(g_i)\mathsf{e_i}-\mathsf{m}^*$. 
     \end{proof}

\begin{cor}\label{nkappa}
For any finite abelian group $G$ we have the inequality
\[\beta_{sep}(G)\leq\left\lfloor\frac{\exp(G)\cdot (\rank(G)+1)}{2}\right\rfloor.\]  
\end{cor}
\begin{proof}
Let $\mathsf{m}$ be a group atom in $\mathcal{B}(g_1,\dots,g_k)$. If $|\mathsf{m}|\leq \max\{\ord(g_i)\mbox{ for } i=1,\dots,k\}$, then $|\mathsf{m}|\leq\frac{\exp(G)\cdot (\rank(G)+1)}{2}$, since $\ord(g_i)\leq \exp(G)$, and $1\leq\rank(G)$.\par

If $|\mathsf{m}|> \max\{\ord(g_i)\mbox{ for } i=1,\dots,k\}$, then $|\mathsf{m}|\leq \frac{\sum_{i=1}^k\ord(g_i)}{2}$ by Lemma \ref{m*geqm} (ii). Lemma \ref{md2} implies that we need to consider only the case $k\leq \rank(G)+1$, while $\ord(g_i)\leq \exp(G)$ by definition. So the maximal length of the group atoms in the monoids given by Lemma \ref{md2} is bounded from above by $\frac{\exp(G)\cdot (\rank(G)+1)}{2}$ in any case. So $\beta_{sep}(G)\leq\left\lfloor\frac{\exp(G)\cdot (\rank(G)+1)}{2}\right\rfloor$, since $\beta_{sep}(G)$ must be an integer.
\end{proof}
\begin{remark} 
The bound for $\beta_{sep}(G)$ in Corollary~\ref{nkappa} is useful when the size of more than half of the cyclic direct summands of $G$ equals the exponent of $G$. Moreover, the inequality in Corollary~\ref{nkappa} is sharp for the groups $C_n^{2s-1}$ for all $n$ and for the groups $C_n^{2s}$ with even $n$ by  
Theorem~\ref{ptlns}.
\end{remark} 

\section{Detecting group atoms}\label{sec:groupatoms}
Next we give a sufficient condition for an element in a block monoid to be a group atom.
\begin{lemma}\label{groupatomL}
Let $\mathsf{m}=[m_1,\dots,m_k]\in\mathcal{B}(g_1,\dots,g_k)$ be an element with $|\supp(\mathsf{m})|=k$, such that 
for a fixed $i\in\{1,\dots,k\}$ and an integer $d$, we have that  
\begin{itemize}
\item[(i)] $d$ 
 divides the $i^{th}$ entry of each atom in $\mathcal{B}(g_1,\dots,g_k)$ whose support is of size strictly smaller than $k$; 
 \item[(ii)] $d$ does not divide $m_i$; 
 \item[(iii)] $|\mathsf{m}|$ is minimal among the lengths of the atoms in $\mathcal{B}(g_1,\dots,g_k)$ with support of size $k$. 
 \end{itemize} 
 Then $\mathsf{m}$ is a group atom in  $\mathcal{B}(g_1,\dots,g_k)$. 
\end{lemma}
\begin{proof}
Suppose for contradiction, that $\mathsf{m}$ can be written as a linear combination of elements of lower length. By the assumptions, the size of their support is strictly smaller than $k$, so their $i^{th}$ entries are divisible by $d$. Of course each linear combination of them does so. This is a contradiction, since $d$ does not divide $m_i$. So $\mathsf{m}$ is indeed a group atom.
\end{proof}

We continue with an elementary number theoretic lemma.
\begin{lemma}\label{lexists}
Let $\alpha,\beta,\gamma$ be positive integers with $\gcd(\alpha,\beta)=1$. Then there exists an $\ell\in\{1,2,\dots,\alpha\gamma-1\}$, for which $\gcd(\ell,\alpha\gamma)=1$ and $\ell\beta\equiv 1$ $\md$ $\alpha$ hold.
\end{lemma}
\begin{proof}
Here $\gcd(\alpha,\beta)=1$ implies that there exists $\hat{\beta}\in\{1,2,\dots,\alpha-1\}$, such that $\hat{\beta}\beta\equiv 1$ $\md$ $\alpha$, (of course, $\gcd(\alpha,\hat{\beta})=1$). Let $p_1,\dots,p_t$ be the distinct prime divisors of $\gamma$ that do not divide $\alpha$. 
By the Chinese Remainder Theorem, the system $x\equiv \hat{\beta}$  $\md$ $\alpha$, $x\equiv 1$  $\md$ $p_i$, $i=1,\dots,t$ has a unique solution $\ell$ in  $\{1,\dots,\alpha p_1\dots p_t-1\}$ contained in $\{1,2,\dots,\alpha\gamma-1\}$. Clearly, $\ell$ is coprime to $\alpha\gamma$ as well. 
\end{proof}

Using Lemma \ref{lexists} we can prove the following result:

\begin{lemma}\label{kn+gcd}
Consider the case when $r=2s$ and $n_{s+1}=\dots=n_1$. Let $\mathsf{m}=[m_1,\dots,m_{2s+1}]\in\mathcal{B}(g_1,\dots,g_{2s+1})$, for elements $g_1,\dots,g_{2s+1}\in G$ satisfying the following conditions:
\begin{itemize}
    \item[(i)] $\ord(g_i)=n_1$ for $i\in\{1,\dots,2s\}$ 
    \item[(ii)] $|\mathsf{m}|>sn_1+\frac{\ord(g_{2s+1})}{p}$, where $p$ is the minimal prime divisor of $\ord(g_{2s+1})$.    
\end{itemize}
Then $\mathsf{m}$ is not a group atom in $\mathcal{B}(g_1,\dots,g_{2s+1})$.
\end{lemma}
\begin{proof}
Suppose for contradiction that $\mathsf{m}$ is a group atom in $\mathcal{B}(g_1,\dots,g_{2s+1})$. 
Set \[d:=\gcd(|\mathsf{m}|,\ord(g_{2s+1}))=\gcd(|\mathsf{m}|-sn_1,\ord(g_{2s+1})).\]  
So $|\mathsf{m}|=sn_1+bd$ for some positive integer $b$ with  $\gcd(b,\frac{\ord(g_{2s+1})}{d})=1$. We use Lemma \ref{lexists} with \[\alpha=\frac{\ord(g_{2s+1})}{d},\quad  \beta=b,\quad  
\gamma= \frac{dn_1}{\ord(g_{2s+1})}.\]  This way we get an $\ell\in\{1,2,\dots,n_1-1\}$, for which $\gcd(\ell,n_1)=1$ and $\ell b\equiv 1$ $\md$ $\frac{\ord(g_{2s+1})}{d}$ hold, implying in particular that \begin{equation}\label{eq:a}
    \ell bd\equiv d \ \md \ord(g_{2s+1}).
\end{equation}
Denote by $m_{\ell,i}$ the unique element in 
$\{0,1,\dots,\ord(g_i)-1\}$ such that $m_{\ell,i}\equiv \ell m_i$ $\md$ $\ord(g_i)$. Observe that $\mathsf{m}_{\ell}:=[m_{\ell,1},\dots m_{\ell,2s+1}]\in\mathcal{B}(g_1,\dots,g_{2s+1})$. Then
\[|\mathsf{m}_{\ell}|=\sum_{i=1}^{2s+1} m_{\ell,i}\equiv \sum_{i=1}^{2s+1} \ell m_i=\ell|\mathsf{m}|=\ell(sn_1+bd)\equiv \ell bd \md\ord(g_{2s+1})\]
implying by \eqref{eq:a} that 
\[|\mathsf{m}_{\ell}|\equiv d\md \ord(g_{2s+1}).\] 
Taking into account 
\[|\mathsf{m}_{\ell}|<\sum_{i=1}^{2s+1}\ord(g_i)=2sn_1+\ord(g_{2s+1})\] 
we get that 
\begin{equation}\label{eq:d}
|\mathsf{m}_{\ell}|\in\left\{d+j\ord(g_{2s+1})\mid j=0,1,\dots,\frac{2sn_1}{\ord(g_{2s+1})}\right\}.\end{equation}
Since $\gcd(\ell,n_1)=1$ we have that there exists an $\hat{\ell}\in\{1,2,\dots,n_1-1\}$, such that $\ell\hat{\ell}\equiv 1$ $\md$ $n_1$ (hence $\ell\hat{\ell}\equiv 1$ $\md$ $\ord(g_{2s+1})$ also holds). 
Moreover, 
\[\hat{\ell}m_{\ell,i}\equiv \hat{\ell}\ell m_i\equiv m_i\md\ord(g_i),\] 
so for each $i\in\{1,\dots,2s+1\}$ there exists an integer $t_i$ such that \[\hat{\ell}m_{\ell,i}-t_i\ord(g_i)=m_i.\] 
This shows that $\mathsf{m}$ can be expressed by using $\mathsf{m}_{\ell}$ and $\mathsf{e_i}$ as follows:
\begin{center}
    $\mathsf{m}=\hat{\ell}\mathsf{m}_{\ell}-\sum_{i=1}^{2s+1}t_i(\ord(g_i) \mathsf{e_i})$.
\end{center}
Since $|\ord(g_i)\mathsf{e_i}|=\ord(g_i)<sn_1+\frac{\ord(g_{2s+1})}{p}<|\mathsf{m}|$, the assumption that $\mathsf{m}$ is a group atom in $\mathcal{B}(g_1,\dots,g_{2s+1})$ implies that 
\begin{equation}\label{eq:e}|\mathsf{m}_{\ell}|\ge |\mathsf{m}|>sn_1+\frac{\ord(g_{2s+1})}{p}.\end{equation} 
Of course, $\mathsf{m}$ can be expressed by using $\mathsf{m}_{\ell}^*$ and $\mathsf{e_i}$ too:
\begin{center}
    $\mathsf{m}=\hat{\ell}(\sum_{i=1}^{2s+1}\ord(g_i) \mathsf{e_i}-\mathsf{m}_{\ell}^*)-\sum_{i=1}^{2s+1}t_i(\ord(g_i) \mathsf{e_i})$.
\end{center}
Again, by the fact that $\mathsf{m}$ is a group atom in $\mathcal{B}(g_1,\dots,g_{2s+1})$ with $|\ord(g_i)\mathsf{e_i}|<|\mathsf{m}|$, we get that $sn_1+\frac{\ord(g_{2s+1})}{p}<|\mathsf{m}|\leq|\mathsf{m}_{\ell}^*|$, hence
\begin{equation}\label{eq:f}
|\mathsf{m}_{\ell}|=2sn_1+\ord(g_{2s+1})-|\mathsf{m}_{\ell}^*|<sn_1+\ord(g_{2s+1})-\frac{\ord(g_{2s+1})}{p}.
\end{equation}
Combining  \eqref{eq:e} and \eqref{eq:f} we get 
\begin{equation}\label{eq:g} 
sn_1+\frac{\ord(g_{2s+1})}{p}<
|\mathsf{m}_{\ell}|
<sn_1+\ord(g_{2s+1})-\frac{\ord(g_{2s+1})}{p}.
\end{equation}
Now \eqref{eq:d} and \eqref{eq:g} imply that 
$|\mathsf{m}_{\ell}|=d+sn_1$, 
showing in turn by \eqref{eq:g} that 
\[\frac{\ord(g_{2s+1})}{p}<d<\ord(g_{2s+1})(1-\frac{1}{p}).\]
However, this is a contradiction, since $d<\ord(g_{2s+1})(1-\frac{1}{p})$ implies that $d$ is a proper divisor of $\ord(g_{2s+1})$, so $\frac{\ord(g_{2s+1})}{p}<d$ can not hold (recall that $p$ is the minimal prime divisor of $\ord(g_{2s+1})$).
This shows that $|\mathsf{m}|$ is not a group atom.
\end{proof}

\begin{ex}
We mention that an element $\mathsf{m}$ satisfying the conditions of Lemma \ref{kn+gcd} may well be an atom in $\mathcal{B}(g_1,\dots,g_{2s+1})$. Indeed, consider for example the group $C_{12}\oplus C_{4}$, and the element  $[11,1,3]\in\mathcal{B}((1,0),(1,1),(0,1))$. In this case $\ord(g_1)=\ord(g_2)=12$, while $\ord(g_3)=4$, and so $\frac{\ord(g_3)}{p}=2$. Here $15=11+3+1>1\cdot12+2=14$, so all the conditions of Lemma \ref{kn+gcd} are satisfied, hence it is not a group atom. Indeed, we can express $[11,1,3]$ as an integral linear combination of such elements of $\mathcal{B}((1,0),(1,1),(0,1))$ that have length strictly smaller than $15$: $[11,1,3]=7[5,7,1]-2[12,0,0]-4[0,12,0]-[0,0,4]$. On the other hand, $[11,1,3]$ is easily seen to be an atom in the monoid $\mathcal{B}((1,0),(1,1),(0,1))$. 
\end{ex}

\section{A general lower bound}\label{sec:lowerbound}
Developing further an idea from the proof of 
Proposition 3.8 of \cite{domokos_abelian}, we shall give a lower bound for the separating Noether number of an abelian group of odd rank $r=2s-1$.  
The following notation is used in Proposition~\ref{divbynr1} and Lemma~\ref{bsgeqsumni}. 
 Denote by $e_1,\dots,e_{s},f_1,\dots,f_{s-1}$ the generators of the direct factors of $G$, (so the order of $e_i$ is $n_i$ for $i=1,\dots,s$ and the order of $f_j$ is $n_{j+s}$ for $j=1,\dots,s-1$). Set $g_1=e_1$, $g_{2i}=e_i+f_i$ and $g_{2i+1}=f_i+e_{i+1}$ for $i=1,\dots,s-1$, and $g_{r+1}=e_s$. The elements $g_1,\dots,g_{r+1}$ 
 are listed in the following table:
\begin{equation}\label{tabl1}
\setcounter{MaxMatrixCols}{15}
\begin{matrix}
    & e_1 & e_2 &  \dots & e_{s-1} & e_s & f_1 & f_2  & \dots & f_{s-2} & f_{s-1} &  \\
    g_1= & (1 & 0 &  \dots & 0 & 0 & 0 & 0 & \dots & 0 & 0) \\
    g_2= & (1 & 0 &  \dots & 0 & 0 & 1 & 0 & \dots & 0 & 0) \\
    g_3= & (0 & 1 &  \dots & 0 & 0 & 1 & 0 & \dots & 0 & 0) \\
    g_4= & (0 & 1 &  \dots & 0 & 0 & 0 & 1 & \dots & 0 & 0) \\
    & & & & & & \dots\\
    g_{r-2}= & (0 & 0 &  \dots & 1 & 0 & 0 & 0 &  \dots & 1 & 0) \\
    g_{r-1}= & (0 & 0 &  \dots & 1 & 0 & 0 & 0 &  \dots & 0 & 1) \\
    g_{r}= & (0 & 0 &  \dots & 0 & 1 & 0 & 0 &  \dots & 0 & 1) \\
    g_{r+1}= & (0 & 0 &  \dots & 0 & 1 & 0 & 0 &  \dots & 0 & 0) \\
\end{matrix}
\end{equation}
\begin{proposition}\label{divbynr1}
Suppose that $r=2s-1$ is odd. Then each coordinate of an atom $\mathsf{m}\in\mathcal{B}(g_1,\dots,g_{r+1})$ (where $g_1,\dots,g_{r+1}$ are listed in \eqref{tabl1}) with $|\supp(\mathsf{m})|<r+1$ is divisible by $n_r$.
\end{proposition}
\begin{proof}
    Let $\mathsf{m}=[m_1,m_2,\dots,m_{r+1}]\in\mathcal{B}(g_1,\dots,g_{r+1})$ be an atom. We have the congruences
\begin{align} \label{zscongr} 
   m_{2i-1}+m_{2i}&\equiv 0 \ \md \ n_i \quad (i=1,\dots,s), 
    \\  \notag
     m_{2i}+m_{2i+1}&\equiv 0 \  \md \ n_{s+i}, \quad (i=1,\dots,s-1)
\end{align}
Since $n_{r}$ divides $n_{i}$ for each $i\in\{1,\dots,r\}$, all the congruences can be seen modulo $n_{r}$, which gives the following chain of congruences:
\begin{equation}\label{chaineq}
    m_1\equiv -m_2\equiv \dots\equiv m_r\equiv-m_{r+1}\hbox{ }\md\hbox{ }n_r
\end{equation}
If $|\supp(\mathsf{m})|<r+1$, then we have an $i_0\in\{1,\dots,r+1\}$, for which $m_{i_0}=0$. Putting this in equation \eqref{chaineq} gives that  $m_i\equiv 0$ $\md$ $n_r$ for $i\in\{1,\dots,r+1\}$.
\end{proof}

\begin{lemma}\label{bsgeqsumni}
Let $G=C_{n_1}\oplus C_{n_2}\oplus \dots\oplus C_{n_r}$, where $2\leq n_r\mid n_{r-1}\mid\dots\mid n_1$. Then for $r=2s-1$, $\beta_{sep}(G)\geq n_1+\dots+n_s$.
\end{lemma}
\begin{proof}
We shall apply Lemma \ref{groupatomL} with 
$k=r+1$, $i=r+1$, $d=n_r$, and 
\begin{center}
    $\mathsf{m}=[n_1-1,1,n_2-1,1,\dots,n_s-1,1]\in\mathcal{B}(g_1,\dots,g_{r+1})$. 
\end{center}
Condition (ii) of Lemma \ref{groupatomL} obviously holds for $\mathsf{m}$, and 
by Proposition~\ref{divbynr1}, condition (i) of Lemma \ref{groupatomL} holds as well. 
Next let $\mathsf{z}=[z_1,\dots,z_{r+1}]\in\mathcal{B}(g_1,\dots,g_{r+1})$ be an element with $|\supp(\mathsf{z})|=r+1$. So $z_i>0$  for each $i\in\{1,\dots,r+1\}$. Moreover, $\mathsf{z}$ fulfills the conditions given in congruences \eqref{zscongr}. So $z_{2i-1}+z_{2i}\equiv 0$ $\md$ $n_i$ with $z_{2i-1}>0$ and $z_{2i}>0$, which yields that $z_{2i-1}+z_{2i}\geq n_i$. So for any element $\mathsf{z}\in\mathcal{B}(g_1,\dots,g_{r+1})$ with $|\supp(\mathsf{z})|=r+1$ we have:
\begin{center}
    $|\mathsf{z}|=(z_1+z_2)+\dots+(z_r+z_{r+1})\geq n_1+\dots+n_s$.
\end{center}
This shows that assumption (iii) of Lemma \ref{groupatomL} is also satisfied by $\mathsf{m}$.  Consequently, $\mathsf{m}$ is a group atom in $\mathcal{B}(g_1,\dots,g_{r+1})$ of length $n_1+\dots+n_s$ by Lemma \ref{groupatomL}. 
\end{proof}

\begin{remark} In the special case, when the size of more
than half of the cyclic direct summands of $G$ equals $\exp(G)$, the inequality in Lemma~\ref{bsgeqsumni} turns out to be sharp, see Theorem \ref{ptlci}. 
\end{remark}

By slight modification of the construction of Lemma \ref{bsgeqsumni}, we are able to give similar results in the case when $r=2s$. In this case let 
$p$ be a prime divisor of $n_r$, and let $g_1=e_1-\frac{n_{s+1}}{p}f_1$, $g_{2i}=e_{i}+f_{i+1}$ and $g_{2i+1}=f_{i+1}+e_{i+1}$ for $i=1,\dots,s-1$, and $g_r=e_s$, $g_{r+1}=f_1$. The elements $g_1,\dots,g_{r+1}$ are listed in the table:
\begin{equation}\label{tabl2}
\setcounter{MaxMatrixCols}{15}
\begin{matrix}
    & e_1 & e_2 &  \dots & e_{s-1} & e_s & f_1 & f_2  & \dots & f_{s-1} & f_{s} &  \\
    g_1= & (1 & 0 &  \dots & 0 & 0 & -\frac{n_{s+1}}{p} & 0 & \dots & 0 & 0) \\
    g_2= & (1 & 0 &  \dots & 0 & 0 & 0 & 1 & \dots & 0 & 0) \\
    g_3= & (0 & 1 &  \dots & 0 & 0 & 0 & 1 &  \dots & 0 & 0) \\
    g_4= & (0 & 1 &  \dots & 0 & 0 & 0 & 0 &  \dots & 0 & 0) \\
    & & & & & & \dots\\
    g_{r-2}= & (0 & 0 &  \dots & 1 & 0 & 0 & 0 &  \dots & 0 & 1) \\
    g_{r-1}= & (0 & 0 &  \dots & 0 & 1 & 0 & 0 &  \dots & 0 & 1) \\
    g_{r}= & (0 & 0 &  \dots & 0 & 1 & 0 & 0 &  \dots & 0 & 0) \\
    g_{r+1}= & (0 & 0 &  \dots & 0 & 0 & 1 & 0 & \dots & 0 & 0) \\
\end{matrix}
\end{equation}
\begin{proposition}\label{divbynr}
Suppose that $r=2s$ is even. Then each coordinate of an atom $\mathsf{m}\in\mathcal{B}(g_1,\dots,g_{r+1})$ (where $g_1,\dots,g_{r+1}$ are listed in \eqref{tabl2}) with $|\supp(\mathsf{m})|<r+1$ is divisible by each prime divisor $p$ of $n_r$.
\end{proposition}
\begin{proof}
Let $\mathsf{m}=[m_1,m_2,\dots,m_{r+1}]\in\mathcal{B}(g_1,\dots,g_{r+1})$ be an atom, and let $p$ be a prime divisor of $n_r$. Then we have the congruences
\begin{align} \notag
    m_{2i-1}+m_{2i}&\equiv 0 \ \md \ n_i \quad (i=1,\dots, s), 
    \\  \notag  
    m_{2i-2}+m_{2i-1}&\equiv 0 \  \md \ n_{s+i}
    \quad (i=2,\dots, s), 
    \\ \label{ns+1} 
    m_{r+1}-m_1\frac{n_{s+1}}{p}&\equiv 0\hbox{ }\md\hbox{ }n_{s+1}.
\end{align}
Since $p$ divides $n_i$ for each $i\in\{1,\dots,r\}$, the congruences can be seen modulo $p$:
\begin{equation}\label{chain}
    m_1\equiv -m_2\equiv\dots \equiv m_{r-1}\equiv -m_r\hbox{ }\md\hbox{ }p
\end{equation}
If $|\supp(\mathsf{m})|<r+1$, then there exists an $i_0\in\{1,\dots,r\}$ for which $m_{i_0}=0$. If $i_0\neq r+1$, then by \eqref{chain} we get that $m_i\equiv 0\md$ $p$ for each $i\in\{1,\dots,r\}$. Moreover, writing $m_1\equiv 0$ $\md$ $p$ in congruence \eqref{ns+1} implies that $m_{r+1}\equiv 0$ $\md$ $n_{s+1}$, so $m_{r+1}\equiv 0$ $\md$ $p$ is also true.\par
If $i_0=r+1$, then $m_{r+1}=0$, so by congruence \eqref{ns+1} we get that $m_1\frac{n_{s+1}}{p}\equiv 0$ $\md$ $n_{s+1}$. From here, $m_1\equiv 0$ $\md$ $p$, which by chain of congruences in \eqref{chain} implies that $m_1\equiv m_2\equiv\dots\equiv m_r\equiv 0$ $\md$ $p$. Since $m_{r+1}=0$, we are done.
\end{proof}

\begin{lemma}\label{bsgeqsumni1}
     Let $G=C_{n_1}\oplus C_{n_2}\oplus \dots\oplus C_{n_r}$, where $2\leq n_r\mid n_{r-1}\mid\dots\mid n_1$ and $r=2s$. Suppose that $p$ is a prime divisor of $n_r$. Then $\beta_{sep}(G)\geq n_1+\dots+n_s+\frac{n_{s+1}}{p}$.
\end{lemma}
\begin{proof}
We shall apply Lemma \ref{groupatomL} again, now with 
$k=r+1$, $i=2$, $d=p$, and 
\begin{center}
    $\mathsf{m}=[n_1-1,1,n_2-1,1,\dots,n_s-1,1,\frac{n_{s+1}}{p}]\in\mathcal{B}(g_1,\dots,g_{r+1})$. 
\end{center}
Condition (ii) of Lemma \ref{groupatomL} obviously holds for $\mathsf{m}$, and by Proposition~\ref{divbynr}, condition (i) of Lemma \ref{groupatomL} holds as well. Next let $\mathsf{z}=[z_1,\dots,z_{r+1}]\in\mathcal{B}(g_1,\dots,g_{r+1})$ be an element with $|\supp(\mathsf{z})|=r+1$. Then $z_{2i-1}+z_{2i}\geq n_i$ for each $i\in\{1,\dots,s\}$, moreover, $z_{r+1}-z_1\frac{n_{s+1}}{p}\equiv 0$ $\md$ $n_{s+1}$, which gives that $z_{r+1}\equiv 0$ $\md$ $\frac{n_{s+1}}{p}$, so $z_{r+1}\geq\frac{n_{s+1}}{p}$. So for any element $\mathsf{z}\in\mathcal{B}(g_1,\dots,g_{r+1})$ with $|\supp(\mathsf{z})|=r+1$ we have:
\begin{center}
$|\mathsf{z}|=(z_1+z_2)+\dots+(z_r+z_{r+1})\geq n_1+\dots+n_s+\frac{n_{s+1}}{p}$  \end{center} 
Hence the assumption (iii) of Lemma \ref{groupatomL} is also satisfied by $\mathsf{m}$.  Consequently, $\mathsf{m}$ is a group atom in $\mathcal{B}(g_1,\dots,g_{r+1})$ of length $n_1+\dots+n_s+\frac{n_{s+1}}{p}$ by Lemma \ref{groupatomL}. 
\end{proof}

\section{Main results and their proofs}\label{sec:results}

\begin{theorem}\label{ptlci}
Assume that $2\le n_{2s-1}\mid n_{2s-2}\mid\dots\mid n_{s+1}\mid n_s=n_{s-1}=\dots=n_1$. 
Then we have 
    $\beta_{sep}(C_{n_1}\oplus\dots\oplus C_{n_{2s-1}})=sn_1$.
\end{theorem}    
\begin{proof}
    By Lemma \ref{bsgeqsumni} we get that $\beta_{sep}(C_{n_1}\oplus\dots\oplus~ C_{n_{2s-1}})\geq n_1+n_2+\dots+n_s=sn_1$, while $\beta_{sep}(C_{n_1}\oplus\dots\oplus C_{n_{2s-1}})\leq\lfloor\frac{2sn_1}{2}\rfloor=sn_1$ by Corollary \ref{nkappa}. So $\beta_{sep}(C_{n_1}\oplus\dots\oplus C_{n_{2s-1}})=sn_1$.
\end{proof}

\begin{theorem}\label{psci}
     Assume that $2\leq n_{2s}\mid n_{2s-1}\mid \dots\mid n_{s+2}\mid n_{s+1}=\dots=n_1$ and that the minimal prime divisor $p$ of $n_1$ divides $n_{2s}$. Then $\beta_{sep}(C_{n_1}\oplus\dots\oplus C_{n_{2s}})=sn_1+\frac{n_1}{p}$.
\end{theorem}
\begin{proof} 
By Lemma \ref{bsgeqsumni1} we have that $\beta_{sep}(C_{n_1}\oplus\dots\oplus C_{n_{2s}})\geq n_1s+\frac{n_1}{p}$. Suppose that $\beta_{sep}(C_{n_1}\oplus\dots\oplus C_{n_{2s}})>n_1s+\frac{n_1}{p}$. This implies that there exist elements $g_1,\dots,g_{2s+1}\in G$ for which there exists a group atom $\mathsf{m}=[m_1,\dots,m_{2s+1}]\in\mathcal{B}(g_1,\dots,g_{2s+1})$  with $|\mathsf{m}|> n_1s+\frac{n_1}{p}$. By Lemma \ref{m*geqm} (ii), the inequality $\ord(g_1)+\dots+\ord(g_{2s+1})\geq2|\mathsf{m}|> 2(n_1s+\frac{n_1}{p})$ also holds. This can only happen, if $\ord(g_i)=n_1$ for each $i\in\{1,\dots,2s+1\}$. So the first condition of Lemma \ref{kn+gcd} is satisfied.\par
Since $p$ is the minimal divisor of both of $n_{2s}$ and $n_1$, the condition (ii) of Lemma \ref{kn+gcd} is also satisfied. So $|\mathsf{m}|$ can not be a group atom in $\mathcal{B}(g_1,\dots,g_{2s+1})$. This contradiction implies that $\beta_{sep}(C_{n_1}\oplus\dots\oplus C_{n_{2s}})=n_1s+\frac{n_1}{p}$.
\end{proof}

\begin{proofof}{Theorem~\ref{ptlns}}
Statement (i) is the special case $n_1=\dots=n_{2s-1}=n$ of 
Theorem~\ref{ptlci}, whereas (ii) is the special case $n_1=\dots=n_{2s}=n$
of Theorem~\ref{psci}. 
\end{proofof}

\begin{remark}   
Another infinite family of abelian groups for which the exact value of the separating Noether number is known is given in \cite[Theorem 3.10]{domokos_abelian}, asserting that $\beta_{sep}(G) \leq \mathsf{d}^*(G)+1$ for all abelian groups $G$, with equality  if and only if $G$ is cyclic or $2=n_{s+1}=$ $\dots$ $= n_r$ where $r=2s-1$ or $r=2s$. Indeed, our Lemma \ref{bsgeqsumni} and Lemma \ref{bsgeqsumni1} implies that for such groups $G$, the inequality $\beta_{sep}(G)\geq \mathsf{d}^*(G)+1$ holds. 
\end{remark}


\begin{thebibliography}{}

\bibitem{cziszter-domokos-geroldinger} K. Cziszter, M. Domokos, A. Geroldinger, The interplay of invariant theory with multiplicative ideal theory and with arithmetic combinatorics, in: Scott T. Chapman, M. Fontana, A. Geroldinger, B. Olberding (Eds.), Multiplicative Ideal Theory and Factorization Theory, Springer-Verlag, 2016, pp. 43-95.
\bibitem{derksen-kemper} H. Derksen and G. Kemper, Computational Invariant Theory, Second Edition, Encyclopaedia of Mathematical Sciences 130, Invariant Theory of Algebraic Transformation Groups VIII, Springer-Verlag, Berlin, Heidelberg, 2015. 
\bibitem{domokos_typical} M. Domokos, Typical separating invariants, Transform. Groups 12 (2007), 49-63.
\bibitem{domokos_abelian} M. Domokos, Degree bounds for separating invariants of abelian groups. Proceedings of the American Mathematical Society, 145:3695–3708, 2017.
\bibitem{domokos:BLMS} M. Domokos, 
 Separating monomials for diagonalizable actions, 
 Bull. London Math. Soc. 55 (2023), 205-223. 
\bibitem{domokos-szabo} M. Domokos and E. Szab\'o, Helly dimension of algebraic groups, J. London Math. Soc. (2) 84 (2011), 19-34. 
\bibitem{gaogeroldinger1} W. Gao and A. Geroldinger, Zero-sum problems and coverings by proper cosets,
European Journal of Combinatorics 24 (2003) 531–549
\bibitem{d^*} W. Gao and A. Geroldinger, Zero-sum sequence problems in finite abelian groups: a survey,
Expo. Math. 24 (2006), no. 4, 337–369.
\bibitem{ghk} A. Geroldinger and F. Halter-Koch, Non-Unique Factorizations. Algebraic, Combinatorial and Analytic Theory. Pure Appl. Math., Vol. 278, Chapman \& Hall/CRC, (2006).
\bibitem{glp} A. Geroldinger, M. Liebmann, and A. Philipp, On the Davenport constant and on the structure of extremal sequences, Period. Math. Hung.  64, 213–225 (2012)
\bibitem{geroldinger-schneider} A. Geroldinger, R. Schneider, 
On Davenport's constant, 
Journal of Combinatorial Theory, Series A 61 (1992), 147-152. 
\bibitem{grynkiewicz} D. J. Grynkiewicz, Structural additive theory, Developments in Mathematics, vol. 30,
Springer, Cham, 2013.
\bibitem{girard} B. Girard, An asymptotically tight bound for the Davenport constant. J. Ec. Polytech. Math. 5, 605–
611 (2018)
\bibitem{kohls-kraft}
M. Kohls and H. Kraft, Degree bounds for separating invariants, Math. Res. Lett. 17 (2010), 1171-1182. 
\bibitem{mr} F. E. B. Martínez and S. Ribas, The \{1, s\}-weighted Davenport constant in $C_n^k$. Integers, v. 22, artigo 36, 2022.
\bibitem{noether} E. Noether, Der Endlichketssatz der Invarianten endlicher Gruppen, Math. Ann. 77 (1916) 89-92.
\bibitem{p-group} J. E. Olson, A combinatorial problem on finite Abelian groups. I, J. Number Theory 1
(1969), 8–10. MR0237641
\bibitem{rank2} J. E. Olson, A combinatorial problem on finite Abelian groups. II, J. Number Theory 1
(1969), 195–199. MR0240200
\bibitem{schefler} B. Schefler, The separating Noether number of abelian groups of rank at most $3$, in preparation. 
\bibitem{schmid} B. J. Schmid, Finite groups and invariant theory, Topics in Invariant Theory, Lecture Notes in Mathematics, vol. 1478, Springer, 1991, pp.~35-66.

\end{thebibliography}
\end{document}